\documentclass[12pt,oneside]{amsart}
\usepackage{geometry}                %
\usepackage{graphicx}
\usepackage{amssymb}
\usepackage{epstopdf}

\usepackage{amsmath,amsthm,amscd,amssymb}
\usepackage{latexsym}
\usepackage[colorlinks,citecolor=red,pagebackref,hypertexnames=false]{hyperref}
\numberwithin{equation}{section}
\theoremstyle{plain}
\newtheorem{theorem}{Theorem}[section]
\newtheorem{lemma}[theorem]{Lemma}

\theoremstyle{definition}

\theoremstyle{remark}
\newtheorem{remark}[theorem]{Remark}

\newtheorem{case[theorem]}{Case}

\def \F {\mathbb{F}}

\date{February 7, 2017}      
\author{P. Birklbauer and A. Iosevich}
\address{Department of Mathematics, University of Rochester, Rochester, NY}
\email{philipp.birklbauer@rochester.edu}
\address{Department of Mathematics, University of Rochester, Rochester, NY}
\email{iosevich@math.rochester.edu}
\thanks{This work was partially supported by the NSA Grant H98230-15-1-0319}
\title{A two-parameter finite field Erd\H os-Falconer distance problem} 
\begin{document}
\maketitle
\begin{abstract} We study the following two-parameter variant of the Erd\H os-Falconer distance problem. Given $E,F \subset {\Bbb F}_q^{k+l}$, 
$l \geq k \ge 2$, the $k+l$-dimensional vector space over the finite field with $q$ elements, let $B_{k,l}(E,F)$ be given by 
$$\{(\Vert x'-y'\Vert, \Vert x''-y'' \Vert): x=(x',x'') \in E, y=(y',y'') \in F; x',y' \in \F_q^k, x'',y'' \in {\Bbb F}_q^l \}.$$ 

We prove that if $|E||F| \geq C q^{k+2l+1}$, then $B_{k,l}(E,F)={\Bbb F}_q \times {\Bbb F}_q$. Furthermore this result is sharp if $k$ is odd. For the case of $l=k=2$ and $q$ a prime with $q \equiv 3 \mod 4$ we get that for every positive $C$ there is $c$ such that
$$ \text{if } |E||F|>C q^{6+\frac{2}{3}}\text{, then } |B_{2,2}(E,F)|>
c q^{2}.$$
\end{abstract}  

\maketitle

\section{Introduction} 

The Erd\H os-Falconer distance problem in ${\Bbb F}_q^d$ is to determine how large $E \subset {\Bbb F}_q^d$ needs to be to ensure that 
$$\Delta(E)=\{||x-y||: x,y \in E\},$$ with $||x||=x_1^2+x_2^2+\dots+x_d^2$, is the whole field ${\Bbb F}_q$, or at least a positive proportion thereof. Here and throughout, ${\Bbb F}_q$ denotes the field with $q$ elements and ${\Bbb F}_q^d$ is the $d$-dimensional vector space over this field. 

The distance problem in vector spaces over finite fields was introduced by Bourgain, Katz and Tao in \cite{BKT04}. In the form described above, it was introduced by the second listed author of this paper and Misha Rudnev (\cite{IR07}), who proved that $\Delta(E)={\Bbb F}_q$ if $|E|>2q^{\frac{d+1}{2}}$. It was shown in \cite{HIKR11} that this exponent is essentially sharp for general fields when $d$ is odd. When $d=2$, it was proved in \cite{CEHIK10} that if if $E \subset {\Bbb F}_q^2$ with $|E| \ge cq^{\frac{4}{3}}$, then $|\Delta(E)| \ge C(c)q$. We do not know if improvements of the $\frac{d+1}{2}$ exponent are possible in even dimensions $\ge 4$. We also do not know if improvements of the $\frac{d+1}{2}$ exponent are possible in any even dimension if we wish to conclude that $\Delta(E)={\Bbb F}_q$, not just a positive proportion. 

In this paper we introduce a two-parameter variant of the Erd\H os-Falconer distance problem. Given $E,F \subset {\Bbb F}_q^{k+l}$, $l \geq k \ge 2$, the $k+l$-dimensional vector space over the finite field with $q$ elements, define $B_{k,l}(E,F)$ by 
$$\{(\Vert x'-y'\Vert, \Vert x''-y'' \Vert): x=(x',x'') \in E, y=(y',y'') \in F; x',y' \in \F_q^k, x'',y'' \in {\Bbb F}_q^l \}.$$ 

This formulation introduces immediate interesting geometric complications. For example, let $k=l=2$, let 
$$ E=\{(x,0,0): ||x||=1\} \ \text{and} \ F=\{(0,0,y): ||y||=1\}.$$ 

Then $B_{2,2}(E,F)=\{(1,1)\}$. However, we are going to see that if $|E||F|$ is sufficiently large, then $B_{k,l}(E,F)={\Bbb F}_q \times {\Bbb F}_q$. Our first result is the following. 

\begin{theorem} \label{main} Let $E,F \subseteq {\Bbb F}_q^{k+l}$, $l\geq k \ge 2$. There is a $C>0$ such that
$$ \text{if } |E||F|> C q^{k+2l+1} \text{ then } B_{k,l}(E,F)=
{\Bbb F}_q \times {\Bbb F}_q.$$

If $k$ is odd, this result is best possible, up to the value of the constant $C$. 
\end{theorem} 

\vskip.125in 

When $k$ is even, we can hope to improve the exponent a bit. We are able to accomplish this in the case $k=l=2$. Our second result is the following. 

\begin{theorem} \label{main2}
Let $q$ a prime with $q \equiv 3 \mod 4$.
For every positive $C$ there is $c$ such that for $E,F \subseteq \F_q^{2+2}$
$$ \text{if } |E||F|>C q^{6+\frac{2}{3}}\text{, then } |B_{2,2}(E,F)|>
c q^{2}.$$

While this result probably is not sharp, we show the exponent cannot go below 6.
\end{theorem}

\section{Proof of Theorem \ref{main}} 

We begin with a quick review of Fourier analytic preliminaries. 

\vskip.125in 

Let $\chi$ be the principal additive character on ${\Bbb F}_q$. Given $f: {\Bbb F}^d_q \to \mathbb{C}$, define 
$$ \widehat{f}(m)=q^{-d} \sum_{x \in {\Bbb F}_q^d} \chi(-x \cdot m) f(x).$$ 

Observe that 
$$ f(x)=\sum_{m \in {\Bbb F}_q^d} \chi(x \cdot m) \widehat{f}(m),$$
$$ \sum_{m \in {\Bbb F}_q^d} {|\widehat{f}(m)|}^2=q^{-d} \sum_{x \in {\Bbb F}_q^d} {|f(x)|}^2$$ and 
$$ \sum_{x \in {\Bbb F}_q^d} \chi(x \cdot m)=0 \ \text{if} \ m \not=\vec{0} \ \text{and} \ q^d \ \text{otherwise}.$$ 

\begin{lemma} \label{kloosterman} Let $S_t^{d-1}=\{x \in {\Bbb F}_q^d: \Vert x\Vert=t \}$, where $\Vert x\Vert =x_1^2+\dots+x_d^2$. If $t \not=0$ and $m \not=\vec{0}$, then 
$$ |\widehat{S}_t^{d-1}(m)| \leq 2q^{-\frac{d+1}{2}}.$$ 
\end{lemma} 

\begin{lemma} \label{spherecount} With the notation above, 
$$ |S_t^{d-1}|=q^{d-1}+O(q^{d-2}).$$ 
\end{lemma} 

\vskip.125in 

For a proof of Lemma \ref{kloosterman} and Lemma \ref{spherecount}, see \cite{IR07}. See also \cite{MMST96} and \cite{IK04}. See \cite{V11} on a spectral graph theory viewpoint on similar phenomena. 

\vskip.125in 

We now move on to the proof of Theorem \ref{main}. Let $E(X), F(Y)$ denote the indicator functions of $E,F$, respectively, where $X=(x',x'')$ and $Y=(y',y'')$. Consider 
\begin{align}
\nonumber &\sum_{\Vert x'-y'\Vert =a; \Vert x''-y''\Vert =b} E(X)F(Y) 
\\ \nonumber &=\sum_{X,Y} S_a^{k-1}(x'-y')S_b^{l-1}(x''-y'') E(X)F(Y)
\\ \nonumber &=\sum_{X,Y,m',m''} \widehat{S}_a^{k-1}(m')\widehat{S}_b^{l-1}(m'')\chi((x'-y') \cdot m') \chi((x''-y'') \cdot m'') E(X)F(Y)
\\ \nonumber &=\sum_{X,Y,m',m''} \widehat{S}_a^{k-1}(m')\widehat{S}_b^{l-1}(m'')\chi((X-Y) \cdot M) E(X)F(Y)
\\ \label{key} &=q^{2(k+l)} \sum_M \widehat{S}_a^{k-1}(m')\widehat{S}_b^{l-1}(m'') \widehat{E}(M)\overline{\widehat{F}(M)}.
\end{align}

We shall now break up the sum into three pieces. The first piece is the sum over $m'=m''=\vec{0}$. The second piece is the sum over $m' \not=\vec{0}, m'' \not=\vec{0}$. The third sum is over $m'=\vec{0}, m'' \not=\vec{0}$ or $m''=\vec{0}, m' \not=\vec{0}$. 

\subsection{The term $m'=\vec{0}, m''=\vec{0}$} 

Plugging this condition into (\ref{key}) we obtain 
\begin{equation} \label{mainterm} |E| |F| |S_a^{k-1}| |S_b^{l-1}| q^{-k-l}.\end{equation} 

\subsection{The term $m' \not=\vec{0}, m'' \not=\vec{0}$}
Using Cauchy-Schwarz we see that
$$
\left( \sum_{m' \not=\vec{0} \neq m''} \!\!\!\! \widehat{S}_a^{k-1}(m')\widehat{S}_b^{l-1}(m'') \widehat{E}(M)\overline{\widehat{F}(M)} \right)^2 
\leq \sum_{m' \not=\vec{0} \neq m''} \!\!\!\! |\widehat{S}_a^{k-1}(m')\widehat{S}_b^{l-1}(m'') \widehat{E}(M)|^2 
\sum_{n' \not=\vec{0} \neq n''} \!\!\!\! |\widehat{F}(N)|^2
$$
Now for the first sum we see by using Lemma \ref{kloosterman} and Plancherel that it is bounded by
$$ \left(2q^{-\frac{k+1}{2}}\right)^2 \left( 2q^{-\frac{l+1}{2}} \right)^2  \sum_M {|\widehat{E}(M)|}^2=16 q^{-(k+l+2)} q^{-k-l}|E|.$$
And again by Plancherel
$$
\sum_{N \neq \vec{0}} |\widehat{F}(N)|^2 \leq q^{-k-l}|F|
$$
Therefore
$$
q^{2(k+l)} \sum_{M \neq \vec{0}} \widehat{S}_a^{k-1}(m')\widehat{S}_b^{l-1}(m'') \widehat{E}(M)\overline{\widehat{F}(M)} \leq 4 q^{\frac{k+l}{2}-1} \sqrt{|E||F|}
$$
\subsection{The term $m' \not=\vec{0}, m''=\vec{0}$} 

We obtain 
\begin{align}
\label{mixedterm}
 q^{2(k+l)} \cdot q^{-l}|S_b^{l-1}| \sum_{m' \not=\vec{0}} \widehat{S}_a^{k-1}(m') \widehat{E}\left(m', \vec{0} \right) \overline{\widehat{F}\left(m', \vec{0}\right)} 
\end{align}

Very similarly to the previous case we see
\begin{align*}
\left(\sum_{m' \not=\vec{0}} \widehat{S}_a^{k-1}(m') |\widehat{E}\left(m', \vec{0} \right) \overline{\widehat{F}\left(m', \vec{0}\right)} \right)^2
&\leq \sum_{m' \not=\vec{0}} {\left| \widehat{S}_a^{k-1}(m') \widehat{E}\left(m', \vec{0}\right)\right|}^2
\sum_{n' \not=\vec{0}} {\left|\widehat{F}\left(n', \vec{0}\right)\right|}^2 
\\
& \leq 4 q^{-k-1} \sum_{m'} {\left|\widehat{E}\left(m', \vec{0}\right)\right|}^2 \sum_{n'} {\left|\widehat{F}\left(n', \vec{0}\right)\right|}^2
\end{align*}

And furthermore we have the following
\begin{lemma}
\label{lemmaEhat(m,0)}
For $E \subset \F_q^{k+l}$ we have
$$
\sum_{m' \in \F_q^k} {\left|\widehat{E}\left(m', \vec{0}\right)\right|}^2
\leq q^{-k-l} {|E|}
$$
\end{lemma}
\begin{proof}
\begin{align*}
\sum_{m' \in \F_q^k} {\left|\widehat{E}\left(m', \vec{0}\right)\right|}^2
& = \sum_{m' \in \F_q^k} q^{-2(k+l)} \sum_{\substack{x',y' \in \F_q^k \\ x'',y'' \in \F_q^l}} \chi((x'-y')m') E(x',x'') E(y',y'') \\
& = q^{-k-2l} \sum_{\substack{x' \in \F_q^k \\ x'',y'' \in \F_q^l}} E(x',x'') \underbrace{E(x',y'')}_{\leq 1}
\\
& \leq q^{-k-l} {|E|}. \\
\end{align*}
\end{proof}

So now we can bound \eqref{mixedterm}.
$$
 q^{2(k+l)} \cdot q^{-l}|S_b^{l-1}| \cdot 2 q^\frac{-k-1}{2} q^{-k-l} \sqrt{|E||F|}=
 2 q^\frac{k-1}{2} |S_b^{l-1}| \sqrt{|E||F|}
$$

Putting everything together we see that 
\begin{equation} \label{almostdone}  \sum_{\Vert x'-y'\Vert =a; \Vert x''-y''\Vert =b} E(X)F(Y)=|E||F| \frac{|S_a^{k-1}|}{q^k} \frac{|S_b^{l-1}|}{q^l}+{\mathcal D}, \end{equation} where 
$$ |{\mathcal D}| \leq 2q^{\frac{k-1}{2}} \sqrt{|E||F|}|S_b^{l-1}|+2q^{\frac{l-1}{2}} \sqrt{|E||F|}|S_a^{k-1}|+4 q^{\frac{k+l}{2}-1}\sqrt{|E||F|}.$$

By a direct calculation (remembering that $l\geq k$) and using Lemma \ref{spherecount}, the right hand side of (\ref{almostdone}) is positive if 
$$ |E||F|> 16 q^{k+2l+1},$$ as desired. 

Finally for the sharpness of this result in the case $k$ odd, we need the following theorem from \cite{HIKR11}.

\begin{theorem}
\label{Thm2.7}
There exists $c > 0$ and $E \subset \F_q^d$, $d$ odd, such that
$$
|E| \geq c q^\frac{d+1}{2} \text{ and } \Delta(E) \neq \F_q.
$$

\end{theorem}

Let $E_1 \subset \F_q^k$ be a set as in theorem above and $E_2 = \F_q^l$. With $E = E_1 \times E_2$ we get $|E| \geq c q^\frac{2l+k+1}{2}$ and $B_{k,l}(E,E)=\Delta(E_1) \times \Delta(E_2) = \Delta(E_1) \times \F_q \neq \F_q \times \F_q$ since $\Delta(E_1) \neq \F_q$. Hence our result is sharp if $k$ is odd.

\section{Proof of Theorem \ref{main2}}

For $a,b \in \F_q$ let
$$
s(a,b):=|\{(x',x'',y',y'') \in E \times F: \Vert x'-y'\Vert =a, \Vert x''-y''\Vert =b \}|.
$$

We observe 
$$
\left( \sum_{a,b \in \F_q} s(a,b) \right)^2 = {|E|}^2{|F|}^2
$$
while at the same time Cauchy-Schwarz yields
$$
\left( \sum_{a,b \in \F_q} s(a,b) \right)^2 \leq B_{k,l}(E,F) \sum_{a,b \in \F_q} s(a,b)^2.
$$
Hence,
\begin{equation}
\frac{{|E|}^2{|F|}^2}{\sum_{a,b \in \F_q} s(a,b)^2} \leq B_{k,l}(E,F)
\label{lowerboundPE}
\end{equation}
so an upper bound on $\sum_{a,b \in \F_q} s(a,b)^2$ will provide a lower bound for $B_{k,l}(E,F)$.
Now
\begin{align*}
 s(a,b)^2 &=  \Big| \big\{(x',x'',y',y'',z',z'',w',w'') \in E \times F \times E \times F : \\ &  \Vert x'-y'\Vert =a=\Vert z'-w'\Vert , \Vert x''-y''\Vert =b=\Vert z''-w''\Vert  \big\} \Big| 
\end{align*}
so
\begin{align}
\label{sums^2}
\sum_{a,b \in \F_q} s(a,b)^2  & =  \Big| \big\{(x',x'',y',y'',z',z'',w',w'') \in E \times F \times E \times F : \\ \nonumber &  \Vert x'-y'\Vert =\Vert z'-w'\Vert , \Vert x''-y''\Vert =\Vert z''-w''\Vert  \big\} \Big| .
\end{align}

We now proceed as in \cite{BHIPR17}. For $\theta, \varphi \in SO_2(\F_q)$ we define $r_{\theta,\varphi}^E: \F_q^2 \times \F_q^2 \to \mathbb{C}$ by the following property.
$$
\sum_{u',u'' \in \F_q^2} r_{\theta,\varphi}^E(u',u'') f(u',u'') = \!\!\!\!\!\! \sum_{x',x'',z',z'' \in \F_q^2} f(x'-\theta z',x''-\varphi z'') E(x',x'') E (z',z'')
$$
for all $f: \F_q^2 \times \F_q^2 \to \mathbb{C}$

By setting $$f(u',u'')=\begin{cases} 1, & \text{if } u'=u'' \\ 0, & \text{otherwise} \end{cases}$$ it is easily seen that $r_{\theta,\varphi}^E$ is well defined and we get 
$$
r_{\theta,\varphi}^E(u',u'') = \left| \{(x',x'',z',z'') \in E \times E: x'-\theta z' = u', x''-\varphi z'' = u''\} \right|
$$
Therefore
\begin{multline}
\label{sumr^2}
\sum_{u',u'' \in \F_q^2} r_{\theta,\varphi}^E(u',u'') r_{\theta,\varphi}^F(u',u'')  =
\\ \left| \{(x',x'',z',z'',y',y'',w',w'') \in E^2 \times F^2: x'-\theta z' = y'-\theta w', x''-\varphi z'' =y''-\varphi w''\} \right|
\end{multline}

With $f(u',u'')=q^{-4}\chi(-u'\cdot m'- u''\cdot m'')$ we can also calculate the Fourier-transform
\begin{align*}
\widehat{r_{\theta,\varphi}^E}(m',m'') & =\sum_{u',u'' \in \F_q^2} r_{\theta,\varphi}^E(u',u'')q^{-4}\chi(-u'\cdot m'- u''\cdot m'') 
\\& = \!\!\!\! \sum_{x',x'',z',z'' \in \F_q^2} q^{-4} \chi(-(x'-\theta z')\cdot m'-(x''-\varphi z'')\cdot m'') E(x',x'') E (z',z'')
\\& = q^{4} \widehat{E}(m',m'') \overline{\widehat{E}(\theta m', \varphi m'')}
\end{align*}

Now our key observation is the following

\begin{lemma}
\label{norm and SOd}
Let $q$ a prime, $q \equiv 3 \mod 4$. Then for $x,y \in \F_q^2 \setminus \{ \vec{0} \}$ we have
$\Vert x \Vert =\Vert y \Vert $ if and only if there is a unique $\theta \in SO_2(\F_q)$ such that $x=\theta y$
\end{lemma}

This observation allows us to make the following connection
$$
\sum_{a,b \in \F_q} s(a,b)^2 \leq \sum_{\substack{u',u'' \in \F_q^2 \\ \theta, \varphi \in SO_2(\F_q)}} r_{\theta,\varphi}^E(u',u'') r_{\theta,\varphi}^F(u',u'')
$$
by comparing (\ref{sums^2}) and (\ref{sumr^2}) and seeing that 
$$
\Vert x' - y' \Vert = \Vert z' - w' \Vert \implies \exists \theta \in SO_2(\F_q): x' - \theta z' = y' - \theta w'.
$$

Now
\begin{align*}
\sum_{U \in \F_q^2\times \F_q^2 } r_{\theta,\varphi}^E(U) r_{\theta,\varphi}^F(U) &= \sum_{U \in \F_q^2\times \F_q^2 } \sum_{M \in \F_q^4} \chi(U M)\widehat{r_{\theta,\varphi}^E}(M) \sum_{N  \in \F_q^4} \chi(U N)\widehat{r_{\theta,\varphi}^F}(N)\\
&= \sum_{M \in \F_q^4} \widehat{r_{\theta,\varphi}^E}(M) \sum_{N  \in \F_q^4} \widehat{r_{\theta,\varphi}^F}(N) \sum_{U \in \F_q^2\times \F_q^2 } \chi(U (N+M)) \\
&= q^4 \sum_{M \in \F_q^4} \widehat{r_{\theta,\varphi}^E}(M) \overline{\widehat{r_{\theta,\varphi}^F}(N)}
\end{align*}

and it remains to find a bound for 
\begin{align}
\nonumber
\sum_{\theta, \varphi \in SO_2(\F_q)} &\sum_{u',u'' \in \F_q^2 } r_{\theta,\varphi}^E(u',u'') r_{\theta,\varphi}^F(u',u'')
= q^{4} \sum_{\theta, \varphi \in SO_2(\F_q)} \sum_{m',m'' \in \F_q^2 } \widehat{r_{\theta,\varphi}^E}(m',m'') \overline{\widehat{r_{\theta,\varphi}^F}(m',m'')} \\
&= q^{12}\sum_{m',m'' \in \F_q^2 }  \sum_{\theta, \varphi \in SO_2(\F_q)}\widehat{E}(m',m'') \overline{\widehat{E}(\theta m', \varphi m'')} \overline{\widehat{F}(m',m'')} \widehat{F}(\theta m', \varphi m'').
\label{sumEhats}
\end{align}

Again we will need to split the sum into three terms

\subsection{The term $m'=\vec{0}, m''=\vec{0}$}
Plugging into \eqref{sumEhats} we get
$$
q^{12} \sum_{\theta, \varphi \in SO_2(\F_q)}  {|\widehat{E}(\vec{0},\vec{0})|}^2 {| \widehat{F}(\vec{0},\vec{0})|}^2=
 q^{-4}{|E|}^2{|F|}^2 {|SO_2(\F_q)|}^2.
$$

\subsection{The term $m' \neq \vec{0}, m'' \neq \vec{0}$}
\begin{align*}
 & q^{12}\sum_{m',m'' \in \F_q^2\setminus \{ \vec{0} \} }    \widehat{E}(m',m'')  \overline{\widehat{F}(m',m'')}  \sum_{\theta, \varphi \in SO_2(\F_q)} \overline{\widehat{E}(\theta m', \varphi m'')} \widehat{F}(\theta m', \varphi m'')
 \\  = &
q^{12}\sum_{a,b \in \F_q \setminus \{ 0 \}} \sum_{\substack{ \Vert m' \Vert=a, \Vert m'' \Vert =b}}    \widehat{E}(m',m'')  \overline{\widehat{F}(m',m'')} \sum_{\theta, \varphi \in SO_2(\F_q)}\overline{\widehat{E}(\theta m', \varphi m'')} \widehat{F}(\theta m', \varphi m'')
\\ = & q^{12}\sum_{a,b \in \F_q \setminus \{ 0 \}} {\left| \sum_{\substack{ \Vert m' \Vert=a, \Vert m'' \Vert =b}}   \widehat{E}(m',m'')  \overline{\widehat{F}(m',m'')} \right|}^2
\end{align*}
where we used Lemma \ref{norm and SOd} in the last step.

We continue with a trivial estimate on one of the inner factors
\begin{align*}
& q^{12}\sum_{a,b \in \F_q \setminus \{ 0 \}} {\left| \sum_{\substack{ \Vert m' \Vert=a, \Vert m'' \Vert =b}}   \widehat{E}(m',m'')  \overline{\widehat{F}(m',m'')} \right|}^2
\\ \leq & q^{12} \sum_{a,b \in \F_q \setminus \{ 0 \}} \sum_{\substack{ \Vert m' \Vert=a, \Vert m'' \Vert =b}}    |\widehat{E}(m',m'')|^2 \sum_{ \Vert n' \Vert=a, \Vert n'' \Vert =b}    |\widehat{F}(n',n'')|^2 
\\ \leq & q^{12} \left( \sum_{a,b \in \F_q \setminus \{ 0 \}} \sum_{\substack{ \Vert m' \Vert=a, \Vert m'' \Vert =b}}    |\widehat{E}(m',m'')|^2 \right) \sum_{n',n'' \in \F_q^2\setminus \{ \vec{0} \} }    |\widehat{F}(n',n'')|^2 
\\ \leq & q^{12} \left( \sum_{m',m'' \in \F_q^2}    |\widehat{E}(m',m'')|^2 \right) \left( \sum_{n',n'' \in \F_q^2}    |\widehat{F}(n',n'')|^2 \right)
\\ = & q^{12} \left( q^{-4} \sum_{u', u'' \in \F_q^2}    |E(u',u'')|^2 \right) \left( q^{-4} \sum_{u', u'' \in \F_q^2}    |E(u',u'')|^2 \right)
\\=  &q^{4} |E| |F|.
\end{align*}

\subsection{The term $m' \neq \vec{0}, m'' = \vec{0}$}
As in the two previous cases we see
\begin{align}
\label{mixedterm2}
 \nonumber& q^{12}\sum_{m' \in \F_q^2 }  \sum_{\theta, \varphi \in SO_2(\F_q)}\widehat{E}(m',\vec{0}) \overline{\widehat{E}(\theta m',\vec{0})} \overline{\widehat{F}(m',\vec{0})} \widehat{F}(\theta m', \vec{0})
\\ & = q^{12} |SO_2(\F_q)|\sum_{m' \in \F_q^2\setminus \{ \vec{0} \} } \widehat{E}(m',\vec{0}) \overline{\widehat{F}(m',\vec{0})}   \sum_{\theta \in SO_2(\F_q)} \overline{\widehat{E}(\theta m',\vec{0})}  \widehat{F}(\theta m', \vec{0})
\end{align}

We will deal with the inner sum first. Let $0 \neq a=\Vert m'\Vert $.
\begin{align}
\label{startmixed}
\nonumber \sum_{\theta \in SO_2(\F_q)} \overline{\widehat{E}(\theta m',\vec{0})}  \widehat{F}(\theta m', \vec{0})
& \leq \sum_{\Vert m \Vert=a} \overline{\widehat{E}(m,\vec{0})}  \widehat{F}(m, \vec{0})
\\
 &\leq \sqrt{\sum_{\Vert m \Vert=a} {|\widehat{E}(m,\vec{0})|}^2  \sum_{\Vert n \Vert=a} {|\widehat{F}(n,\vec{0})|}^2 }
\end{align}

\begin{lemma}
\label{lemmaEhat(m,0)sphere}
For $E \subset \F_q^2$, $0 \neq a \in \F_q$ we get
$$
\sum_{\Vert m \Vert=a} | \widehat{E}(m,\vec{0})|^2 \leq 3^{\frac{1}{2}} q^{-6} \vert E \vert^\frac{3}{2} 
$$
\end{lemma}

\begin{proof}

With the notation introduced in Lemma \ref{kloosterman} and $g: \F_q^2 \to \mathbb{C}$ where $g(m)=\overline{\widehat{E}(m,\vec{0}) S_a(m)}$. we can write this as
\begin{align*}
\sum_{m \in \F_q^2} \widehat{E}(m,\vec{0}) S_a(m) g(m)
&= \sum_{m \in \F_q^2} q^{-4}\sum_{x',x'' \in \F_q^2} \chi(-x'\cdot m) E(x',x'') S_a(m) g(m) \\
&=q^{-2} \sum_{x' \in \F_q^2} \left(\sum_{x'' \in \F_q^2} E(x',x'') \right) \widehat{S_a g}(x')
\end{align*}

Using H\"older's Inequality with $q=\frac{4}{3}, r=4$ we can bound this by
\begin{align}
\label{Hoelder}
\leq q^{-2} \left( \sum_{x' \in \F_q^2} \left(\sum_{x'' \in \F_q^2} E(x',x'') \right)^\frac{4}{3} \right)^\frac{3}{4} \left( \sum_{x' \in \F_q^2}  \vert \widehat{S_a g}(x') \vert^4 \right)^\frac{1}{4}
\end{align}

We will first find an estimate for the latter factor. By using the definition of the Fourier transform we get:
\begin{align}
\sum_{x' \in \F_q^2}  \vert \widehat{S_a g}(x') \vert^4
=q^{-6} \sum_{\substack{u,v,u',v' \in S_a\\ u+v=u'+v'}} g(u)g(v)\overline{ g(u')g(v') }
\label{Fefferman_setup}
\end{align}
Here we use the Fefferman trick. For fixed $u,v \in S_a$, $u \neq -v$ we want to find $u',v' \in S_a$ such that $u+v=u'+v'$. In other words we want to find $u' \in S_a$ such that $(u+v-u') \in S_a$, so $u'$ is in the intersection of the circles $\{x \in \F_q^2: \Vert x\Vert=a \}$ and $\{x \in \F_q^2: \Vert x-(u+v)\Vert=a \}$ which has at most two solutions as the circles are not identical $u+v \neq 0$. But we already know two solutions, namely $u$ and $v$. So either $u'=u$ and $v'=v$ or $u'=v$ and $v'=u$. If $u = -v$ we get $u' \in S_a$ and $v'=-u'$.
Therefore (and by noting that $g(-u)=\overline{g(u)}$) we can write \eqref{Fefferman_setup} as
\begin{align*}
 &q^{-6} \left( \sum_{u,v \in S_a} 2 g(u)g(v)\overline{ g(u)g(v) } + \sum_{u,u' \in S_a} g(u)g(-u)\overline{ g(u')g(-u') } \right) \\
= & 3 q^{-6} \sum_{u,v \in S_a} \vert g(u) \vert^2 \vert g(v) \vert^2 \\
= & 3 q^{-6} \left( \sum_{u \in S_a} \vert g(u) \vert^2 \right)^2 \\
= & 3 q^{-6} \left(\sum_{\Vert u \Vert=a} | \widehat{E}(u,\vec{0})|^2 \right)^2
\end{align*}
The other factor of \eqref{Hoelder} can be dealt with as follows
\begin{align*}
\left( \sum_{x' \in \F_q^2} \! \! \left(\sum_{x'' \in \F_q^2} E(x',x'') \right)^\frac{4}{3} \right)^\frac{3}{4} \! \! \! \!
\leq \left(\sum_{x' \in \F_q^2}
  \! \!  \left( \sum_{x'' \in \F_q^2} E(x',x'') \right) \! \!
 \left( \sum_{x'' \in \F_q^2} E(x',x'') \right)^\frac{1}{3} \right)^\frac{3}{4}
 \!\!\!
\leq q^{\frac{1}{2}} \vert E \vert^\frac{3}{4}
\end{align*}

Therefore we have
\begin{align*}
\sum_{\Vert m \Vert=a} | \widehat{E}(m,\vec{0})|^2 \leq 3^{\frac{1}{4}} q^{-2} q^{\frac{1}{2}} \vert E \vert^\frac{3}{4} q^{-\frac{3}{2}} \left(\sum_{\Vert m \Vert=a} | \widehat{E}(m,\vec{0})|^2\right)^\frac{1}{2}
\end{align*}
so
\begin{align*}
\sum_{\Vert m \Vert=a} | \widehat{E}(m,\vec{0})|^2 \leq 3^{\frac{1}{2}} q^{-4} q \vert E \vert^\frac{3}{2} q^{-3} = 3^{\frac{1}{2}} q^{-6} \vert E \vert^\frac{3}{2} 
\end{align*}

\end{proof}

Continuing from \eqref{mixedterm2} and using \eqref{startmixed} and Lemma \ref{lemmaEhat(m,0)sphere} we see

\begin{align*}
& q^{12} |SO_2(\F_q)|\sum_{m' \in \F_q^2\setminus \{ \vec{0} \} } \widehat{E}(m',\vec{0}) \overline{\widehat{F}(m',\vec{0})}   \sum_{\theta \in SO_2(\F_q)} \overline{\widehat{E}(\theta m',\vec{0})}  \widehat{F}(\theta m', \vec{0}) \\
\leq &  q^{12} |SO_2(\F_q)| \sum_{m' \in \F_q^2\setminus \{ \vec{0} \} } \widehat{E}(m',\vec{0}) \overline{\widehat{F}(m',\vec{0})} \cdot 3^{\frac{1}{2}} q^{-6} {|E|}^\frac{3}{4} {|F|}^\frac{3}{4}  
\end{align*}

Finally we need to deal with
\begin{align*}
\sum_{m' \in \F_q^2\setminus \{ \vec{0} \} } \widehat{E}(m',\vec{0}) \overline{\widehat{F}(m',\vec{0})}
& \leq \sqrt{\sum_{m' \in \F_q^2\setminus \{ \vec{0} \} } {|\widehat{E}(m',\vec{0})|}^2 \sum_{m' \in \F_q^2\setminus \{ \vec{0} \} } {|{\widehat{F}(m',\vec{0})}|}^2 } \\
&\leq \sqrt{q^{-8} |E| |F|}
\end{align*}

Putting those results together we find that \eqref{mixedterm2} is bounded by
\begin{align*}
 C q^{12} q q^{-4} \sqrt{|E||F|} q^{-6} {|E|}^\frac{3}{4} {|F|}^\frac{3}{4} 
=C q^3 {|E|}^\frac{5}{4}{|F|}^\frac{5}{4}.
\end{align*}

So we can bound the whole sum \eqref{sumEhats} by 
$$
q^{4} |E||F| + C q^3 \big(|E||F|\big)^\frac{5}{4}  +  q^{-4}|E|^2|F|^2 |SO_d(\F_q)|^2.
$$
Therefore we get from \eqref{lowerboundPE}
$$
\min \left\lbrace \frac{|E||F|}{3 q^{4}},\frac{\big(|E||F|\big)^\frac{3}{4} }{3 C q^3},\frac{q^{4}}{3 |SO_2(\F_q)|^2} \right\rbrace \leq P(E).
$$
Hence it is enough that
$$
c q^2 \leq \frac{\big(|E||F|\big)^\frac{3}{4} }{3 C q^3} \iff c q^5 \leq \big(|E||F|\big)^\frac{3}{4}  \iff c q^{\frac{20}{3}} \leq |E||F|
$$
since in this case also
$$
\frac{|E||F|}{3 q^{4}} \geq \frac{c q^\frac{20}{3}}{q^4} \geq c q^2.
$$

\begin{remark}[Sharpness of results]

Let $p$ a prime, with $p \equiv 3 \mod 4$. Consider 
$E=\F_p^2 \times L$, where 

$$ L=\{(a,0): a \in \lbrace 0,\dots,p^{1-\varepsilon} \rbrace \},$$

Then $|E| = p^{3-\varepsilon}$ and $|\Delta(L)| = 2p^{1-\varepsilon}$, so $|B(E,E)|=o(p^2)$. Hence the $6+\frac{2}{3}$ exponent in Theorem \ref{main2} is potentially not best possible, but we definitely cannot go below $6$. 
\end{remark}

\end{document}